\documentclass{article}
\usepackage{amsfonts}
\usepackage{amssymb}
\usepackage{mathrsfs}
\usepackage{amsthm}
\usepackage{latexsym, amsfonts, amsmath, amssymb, amscd, graphicx}
\usepackage{hyperref}
\newtheorem{theorem}{Theorem}[section]
\newtheorem{definition}[theorem]{Definition}
\newtheorem{lemma}[theorem]{Lemma}

\newtheorem{corollary}[theorem]{Corollary}

\title{\bf On the $n$-dimensional Blaschke-Lebesgue Problem
 }

\author{\\
\\{\bf Botao \ Duan$^{1,2}$, \quad Deyan\ Zhang$^{1,}$\thanks{The corresponding author, email: zhangdy8005@126.com }}
\\\
1. School of Mathematics and Statistics,\ \ HuaiBei Normal University, \\
Huaibei, 235025, P. R. China.\\
2. Nanjing University of Finance, Economics Hongshan College,
\\
Nanjing, 211300, P. R. China, \  email: duanbotao199202@163.com
}
\date{}
\begin{document}
	
\maketitle
	
{\bf Abstract} The classical Blaschke-Lebesgue problem seeks
the convex body of minimal volume among all convex bodies of prescribed constant width.
In this paper,
we investigate the necessary conditions for volume minimizers of the $n$-dimensional
Blaschke-Lebesgue problem via the inverse Gauss map and the
inverse Weingarten transformation.
We first introduce a generalized support function for hypersurfaces induced by the inverse Gauss map,
characterize the Wigner caustic of constant width bodies,
and accordingly construct convex bodies of constant width
from prescribed functions. Building upon there results and applying the Brunn-Minkowski inequality,
we establish a high-dimensional necessary condition: for almost every antipodal pair of
 a volume minimizer, one point is
non-regular; at the other point, the minimum principal curvatures equals the reciprocal of the constant width.
As an application, we provide a new proof for the two-dimensional Blaschke-Lebesgue theorem.
Furthermore, we study the dimension-raising construction of Lachand-Robert and Oudet,
and prove that every convex body produced by this construction satisfies our necessary condition.
Finally, we put forward a conjecture stating that the sequence obtained by iterating this construction
starting from a line segment yields minimizers for the Blaschke-Lebesgue problem for all dimensions.
\\\\
	{\bf 2000 MR Mathematics Subject Classification:} {52A20, 52A40}

\baselineskip 15pt
	
\section{Introduction}
\setcounter{equation}{0}
	
 Let $\mathbb{R}^n$ denote the $n$-dimensional Euclidean space
 equipped with the standard inner product $\langle \cdot, \cdot \rangle$,
 let $S^{n-1}$ stand for the $(n-1)$-dimensional unit sphere
in $\mathbb{R}^n$.
Given a convex body $K\subset\mathbb{R}^n$, the width of $K$ in a direction
$u \in S^{n-1}$ is defined as the Euclidean distance between the two supporting hyperplanes of $K$
 orthogonal to $u$. A convex body is said to be of {\it constant width}
 if its width is the same in every direction.
 The closed ball is the most
 elementary example of a body of constant width.
  In the planar setting,
 Reuleaux constructed a famous constant width set by
 intersecting three closed disks of radius equal to the side length of
 an equilateral triangle, each disk centered at one vertex of the triangle.
 This yields the classical Reuleaux triangle.
 The construction was later generalized to odd-sided polygons, producing planar
 constant width bodies bounded by circular arcs of constant curvature,
 commonly called Reuleaux polygons. Following a related construction
 principle, Meissner constructed
 two distinct types of constant width bodies in $\mathbb{R}^3$ based on a regular tetrahedron;
 these objects are now known as Meissner bodies, see \cite{M-1911, M-1912}.
 In 2007, Lachand-Robert and
 Oudet \cite{LO-2007} introduced a dimension-raising method that produces
 an $n$-dimensional body of constant width from an $(n-1)$-dimensional one;
 this method reveals an intrinsic relation between the Reuleaux triangle and Meissner bodies.

It is well-known that among all $n$-dimensional
 bodies of constant width $\omega > 0$,
 the $n$-dimensional closed ball maximizes the volume.
 Under the same width constraint, the volume of bodies of constant
 width admits a positive infimum. This leads to a celebrated geometric problem:
 among all n-dimensional bodies of constant
 width with prescribed width, which shape attains the minimal volume?
 This is the famous {\it Blaschke-Lebesgue problem}.
 In dimension two, Blaschke \cite{B-1915} and
 Lebesgue \cite{L-1914} independently proved that the Reuleaux triangle
 is the volume minimizer, a result now called the Blaschke-Lebesgue theorem.
 Since then, many authors have provided
 various alternative proofs (see, e.g., references \cite{B-2025, E-1952, F-1927, G-1996, H-2002,
M-2009, YB-1961}). For dimensions $n \geq 3$,
however, the problem remains open to the present day.
In 1934, Bonnesen and Fenchel
\cite{B-1934} conjectured that Meissner bodies are volume minimizers within the class of three-dimensional convex bodies of constant width.
The problem has subsequently been investigated from a variety of analytical and numerical perspectives;
consult \cite{AG-2011,AG-2026, B-2024, B-2025, B-2026,
BH-2012, BL-2007, C-1996, H-2024}. It is worth noting that
Anciaux and Guilfoyle \cite{AG-2011} proved some necessary conditions
on the boundary of extremal bodies in $\mathbb{R}^3$ with respect to
the Blaschke-Lebesgue problem; see Theorem A below.

{\bf Theorem A}(\cite{AG-2011})
Let $B$ be a local minimizer of the Blaschke-Lebesgue problem in $\mathbb{R}^n$,
with constant width $\omega$. Then the smooth parts of its boundary have their smaller
principal curvature constant and equal to $\frac{1}{\omega}$.

In this paper, we investigate the necessary conditions for the $n$-dimensional
Blaschke-Lebesgue problem by means of the inverse Gauss map. We first
introduce the inverse Gauss map and the induced inverse Weingarten
transformation for convex bodies, and analyze their
geometric properties. Subsequently, we extend the support function of
a convex body to a generalized support
function for hypersurfaces determined by the
inverse Gauss map (see Definition \ref{d2.2} below).
Using this generalized support function, we characterize the Wigner caustic
of constant width bodies, and thereby construct convex bodies of constant width
from prescribed functions. Building upon there results and applying the Brunn-Minkowski inequality,
we generalize Theorem A to the $n$-dimensional setting and derive the following
necessary condition for the $n$-dimensional
Blaschke-Lebesgue problem:

{\bf Necessary condition for the $n$-dimensional Blaschke-Lebesgue problem}
Let $K \in \mathcal{K}^n(\omega)$
be a volume minimizer for the $n$-dimensional Blaschke-Lebesgue problem.
Then, for almost every pair of antipodal points of $K$, one point is
non-regular; at the other point, the minimum principal curvatures equals $\frac{1}{\omega}$.

In this paper, a point is said to be non-regular if the inverse Weingarten
transformation fails to have full rank at that point; otherwise,
it is said to be regular. Building upon the necessary condition above and
combining it with results due to Firey \cite{F-1960}, we give
a new proof of the two-dimensional Blaschke-Lebesgue theorem. Furthermore,
we investigate the dimension-raising construction for constant width bodies introduced
by Lachand-Robert and Oudet \cite{LO-2007},
and show that every constant width body produced by
this construction satisfies the necessary condition derived above.
For example, starting from a line segment of length $\omega$
(regarded as a one-dimensional constant width body),
the dimension-raising construction yields the Reuleaux triangle of width $\omega$;
starting from the Reuleaux triangle of width $\omega$,
it yields the Meissner body of the same width. Accordingly,
we conjecture that for $n \geq 3$, the Meissner bodies
and the family of higher-dimensional constant width bodies obtained
 via the dimension-raising construction are minimizers for the Blaschke-Lebesgue problem.

\section{Inverse Gauss Map and Inverse Weingarten Transformation of Convex Bodies}\label{s2}
\setcounter{equation}{0}
Let $\mathcal{K}^n$ denote the collection of all strictly convex bodies in $\mathbb{R}^n$.
For any convex body $K \in \mathcal{K}^n$, we write $\partial K$ for its boundary.
\vskip.3cm
\begin{definition}\label{d2.1} Let $K \in \mathcal{K}^n$. For each unit direction $u \in S^{n-1}$,
define
$$H(K, u)
= \{ p \mid \langle u, p \rangle = \max_{x \in K} \langle u, x \rangle, \; p \in \mathbb{R}^n \}.$$
Then $H(K, u)$ is called the supporting hyperplane of $K$ in the direction $u$.
\end{definition}
\noindent By strict convexity of $K$, the intersection $H(K, u) \cap K$ is a singleton.
Its unique element is referred to as the support point of $K$ in the direction $u$.
Accordingly,
$$\partial K = \bigcup_{u \in S^{n-1}} (H(K, u) \cap K),$$
i.e., $\partial K$ is precisely the set consisting of all support points of $K$.

\vskip.3cm
\begin{definition}\label{d2.2}
Let $K \in \mathcal{K}^n$. A map $G_K^r : S^{n-1} \to \partial K$ is said to be
the {\it inverse Gauss map} of the convex body $K$ if, for every $u \in S^{n-1}$,
$G_K^r(u)$ is the support point of $K$ in the direction $u$.
\end{definition}
\noindent {\it Remark 1.}~
The following properties follow immediately from Definition \ref{d2.2}:
\begin{enumerate}
\item The map $G_K^r$ is continuous and $G_K^r(S^{n-1}) = \partial K$.
\item In general, $G_K^r$ need not be injective. If a hypersurface
  $\Sigma \subset \partial K$ is regular, then for every $p \in \Sigma$, there
  exists a unique $u \in S^{n-1}$ such that $p = G_K^r(u)$, and this $u$ is precisely
  the outward unit normal vector of $\Sigma$ at $p$. Let $V \subset S^{n-1}$ satisfy
  $G_K^r(V) = \Sigma$. Then the restriction $G_K^r|_V : V \to \Sigma$ is bijective.
  Its inverse $(G_K^r|_V)^{-1} : \Sigma \to S^{n-1}$ is the Gauss map of
  the hypersurface $\Sigma$, which we denote by $G_K|_V$.
\end{enumerate}
Recall the definition of the support function $h_K(u)$ of a convex body $K$:
$$
h_K(u) = \max_{x \in K} \langle u, x \rangle.
$$
From the definition of the inverse Gauss map $G_K^r$,
the support function of $K$ admits the equivalent expression
\begin{equation}\label{e2.1}
h_K(u) = \langle G_K^r(u), u \rangle.
\end{equation}

For any $u \in S^{n-1}$, let $T_u S^{n-1}$ stand for the tangent space
of $S^{n-1}$ at $u$. Choose an orthonormal frame $\{u; e_1, e_2, \ldots, e_{n-1}\}$
for $T_u S^{n-1}$. Let $v$ be an arbitrary point in a neighbourhood of $u$,
with local coordinates $(x_1, x_2, \ldots, x_{n-1})$
adapted to the frame $\{u; e_1, e_2, \ldots, e_{n-1}\}$, so that
$v(0, 0, \ldots, 0) = u$.
Taking partial derivatives of formula \eqref{e2.1} with respect to $x_i$ gives
$$
\frac{\partial h_K(v)}{\partial x_i}
= \left\langle \frac{\partial G_K^r(v)}{\partial x_i}, v \right\rangle
+ \left\langle G_K^r(v), \frac{\partial v}{\partial x_i} \right\rangle
  = \left\langle \frac{\partial G_K^r(v)}{\partial x_i}, v \right\rangle
 +\left\langle G_K^r(v), e_i \right\rangle,
  \quad i = 1, \ldots, n-1.
$$
Letting $v \to u$, and recalling that
$G_K^r(v) \in \partial K$ and $dG_K^r(v) = \sum_{i=1}^{n-1} \frac{\partial G_K^r(v)}{\partial x_i} dx_i$
 is tangent to $\partial K$, we have
$$
\left\langle \frac{\partial G_K^r(v)}{\partial x_i}\Big|_{v=u},u \right\rangle = 0.
$$
Consequently,
 $$
 \frac{\partial h_K}{\partial x_i} \Big|_{u} = \langle G_K^r(u), e_i \rangle, \quad i = 1, \ldots, n-1.
 $$
Since $\{e_1, e_2, \ldots, e_{n-1}, u\}$ forms an orthonormal basis of $\mathbb{R}^n$,
the orthogonal decomposition of $G_K^r(u)$ yields
\begin{equation*}
G_K^r(u)
= \sum_{i=1}^{n-1} \langle G_K^r(u), e_i \rangle e_i + \langle G_K^r(u), u \rangle u
= \sum_{i=1}^{n-1} \frac{\partial h_K(u)}{\partial x_i} e_i + h_K(u) u.
\end{equation*}
that is,
\begin{equation}\label{e2.2}
G_K^r(u)
= \sum_{i=1}^{n-1} \frac{\partial h_K(u)}{\partial x_i} e_i + h_K(u) u.
\end{equation}
It may be verified that formula \eqref{e2.2} is
independent of the choice of local orthonormal frame.

Let $K$ be a strictly convex body. Then its support function $h_K(u)$ belongs to
$C^{1,1}(S^{n-1})$, meaning that its first-order derivatives satisfy a Lipschitz condition;
see \cite{H-2006}.
By Rademacher's theorem, $h_K(u)$ is twice differentiable almost everywhere
on $S^{n-1}$. Therefore, for almost every $u\in S^{n-1}$, the inverse Gauss map
$G_K^r$ induces the {\it inverse Weingarten transformation}
$\mathcal{W}_K^r : T_u S^{n-1} \to T_{G_K^r(u)} \partial K$, defined by
$$
\mathcal{W}_K^r(du) = dG_K^r(u).
$$

Assume that $h_K(u)$ is twice differentiable at $u \in S^{n-1}$,
Differentiating $G_K^r(u)$ given by \eqref{e2.2}, we obtain
$$
dG_K^r(u) = \sum_{i=1}^{n-1}\frac{\partial}{\partial x_i}
\left(\sum_{j=1}^{n-1} \frac{\partial h_K(u)}{\partial x_j} e_j + h_K(u) u\right)dx_i.
$$
Employing the spherical connection formulas
$\frac{\partial e_i}{\partial x_i} = -\delta_{ij} u$ and
$\frac{\partial u}{\partial x_i} = e_i$,
expansion yields the component expression
$$
\frac{\partial}{\partial x_i} \left( \sum_{j=1}^{n-1} \frac{\partial h_K(u)}{\partial x_j} e_j + h_K(u) u \right)
= \sum_{j=1}^{n-1} \left( \frac{\partial^2 h_K(u)}{\partial x_i \partial x_j} + \delta_{ij} h_K(u) \right) e_j.
$$
Substituting this identity into the expression for $dG_K^r(u)$, we arrive at
\begin{equation}\label{e2.3}
dG_K^r(u) = \sum_{i,j=1}^{n-1} (h_{ij} + \delta_{ij} h_K(u)) e_j\, dx_i,
\end{equation}
where $h_{ij} = \frac{\partial^2 h_K(u)}{\partial x_i \partial x_j}$.

Let $W_K^r(u)$ be the matrix of the inverse Weingarten transformation
$\mathcal{W}_K^r$ with respect to the orthonormal basis $\{e_1, \ldots, e_{n-1}\}$. Then
$$
dG_K^r(u) = (e_1, e_2, \ldots, e_{n-1}) \, W_K^r(u) \, (dx_1, dx_2, \ldots, dx_{n-1})^T,
$$
where $A^T$ stands for the transpose of a matrix $A$.
A tangent vector $du \in T_u S^{n-1}$ admits the representation
$$
du = (e_1, e_2, \ldots, e_{n-1}) (dx_1, dx_2, \ldots, dx_{n-1})^T.
$$
From formula \eqref{e2.3}, we arrive at
$$
W_K^r(u) = \operatorname{Hess}(h_K(u)) + h_K(u) I,
$$
where $I$ denotes the $(n-1) \times (n-1)$ identity matrix.

If the inverse Weingarten transformation $\mathcal{W}_K^r$ has full rank,
then its matrix $W_K^r(u)$ is invertible, and coincides with the inverse of the usual
Weingarten matrix taken with respect to the orthonormal basis
$\{e_1, \ldots, e_{n-1}\}$ in $T_{G_K^r(u)}\partial K$.
Furthermore, the eigenvectors of
$\mathcal{W}_K^r(u)$
correspond to the principal directions of the boundary hypersurface $\partial K$ at $G_K^r(u)$,
denoted by $\xi_i \in T_u S^{n-1}, \quad i = 1, \ldots, n-1$.
The eigenvalues of $\mathcal{W}_K^r(u)$ are exactly the reciprocals of the principal curvatures
along the associated principal directions, i.e. the principal radii of curvature, written
$\rho_i(K, u), \quad i = 1, \ldots, n-1$.
Both the principal directions and the principal curvatures are independent
of the choice of local orthonormal frame. For the convex hypersurface $\partial K$,
one always has
$\rho_i(K, u) \geq 0$, $ i = 1, \ldots, n-1$;
hence $W_K^r(u)$ is positive semidefinite.

Since the support function $h_K$ is twice differentiable almost everywhere
on $S^{n-1}$, the matrix $W_K^r$ is defined almost everywhere on
$S^{n-1}$. The transformation $\mathcal{W}_K^r$ is the tangent map of the inverse Gauss map
$G_K^r $, and its determinant $\det(W_K^r)$ is the pullback
factor of the volume element from the unit sphere to the boundary of the convex body.
Accordingly, the surface area measure $A$ and the $n$-dimensional volume measure $V$ of
$K$ can be expressed in terms of $\det(W_K^r)$ as
\begin{equation}\label{e2.4}
A(\partial K) = \int_{S^{n-1}} \det(W_K^r) \, du
= \int_{S^{n-1}} \det\left(\operatorname{Hess}(h_K) + h_K I\right) \, du,
\end{equation}
\begin{equation}\label{e2.5}
V(K) = \frac{1}{n} \int_{u \in S^{n-1}} h_K \det(\operatorname{Hess}(h_K) + h_K I) \, du.
\end{equation}

\section{Generalized Support Functions
of Hypersurfaces
and the Wigner Caustic
of Constant Width Bodies
}\label{s3}
\setcounter{equation}{0}

Given any function $h \in C^{1,1}(S^{n-1})$, in analogy with formula \eqref{e2.2},
we construct a mapping $\gamma:S^{n-1}\rightarrow \mathbb{R}^n$ induced by $h$ via
\begin{equation}\label{e3.1}
\gamma(u) = \sum_{i=1}^{n-1} \frac{\partial h}{\partial x_i} e_i + h(u)u, \qquad u \in S^{n-1},
\end{equation}
where $\{x_1, x_2, \ldots, x_{n-1}\}$ are local coordinates on the tangent space $T_u S^{n-1}$
 associated with the orthonormal frame $\{u; e_1, e_2, \ldots, e_{n-1}\}$.
 The map $\gamma(u)$ is independent of the choice of local frame.

 Let $\Gamma = \gamma(S^{n-1})$ be the hypersurface in $\mathbb{R}^n$
 determined by $\gamma$. Since $h \in C^{1,1}(S^{n-1})$, the set $\Gamma$ is closed
 and bounded. Nevertheless, $\Gamma$ need not be a convex hypersurface and may even
possess self-intersections.
Accordingly, $\Gamma$ cannot, in general, be realized as
 the boundary of some convex body. In the present paper, we call the function $h$ appearing in \eqref{e3.1}
 the {\it generalized support function} of the hypersurface $\Gamma$.
 Observe that $h$ still satisfies the identity $h(u) = \langle \gamma(u), u \rangle$.
 In this sense, the generalized support function extends the classical notion of
 the support function from convex bodies to
 hypersurfaces in $\mathbb{R}^n$.

By an argument analogous to the derivation of equation \eqref{e2.3},
one may verify that, if $\gamma$ is twice differentiable at $u$,
then the tangent space $T_{\gamma(u)}\Gamma \subset T_u S^{n-1}$ satisfies the
orthogonality relation $u \perp T_{\gamma(u)} \Gamma$, and the identity
$h(u) = \langle \gamma(u), u \rangle$ continues to hold.
In essence, the map $\gamma : S^{n-1} \to \mathbb{R}^n$
plays the role of an inverse Gauss map induced by $h$,
and we write $G^r(h, u)$ (or $G_{\Gamma}^r(u)$) for $\gamma(u)$.
Here $h(u)$ is the generalized support function of $\Gamma$,
also denoted $h_{\Gamma}(u)$.
In particular, if $K \in \mathcal{K}^n$ is a strictly convex body and
$\partial K = \Gamma$, then $h_K(u) = h_{\Gamma}(u)$.
Moreover, the formulae for the inverse Weingarten matrix and its eigenvalues (the principal radii of curvature)
for convex bodies admits a natural extension to the possibly non-convex hypersurface defined
by a generalized support function.
The resulting generalized inverse Weingarten
matrix and its eigenvalues are denoted, respectively,
by $W_{\Gamma}^r(u)$ and $\rho_i(\Gamma, u), i = 1, \dots, n-1$,
where
\begin{equation}\label{e03.2}
W_{\Gamma}^r(u) = \operatorname{Hess}(h_{\Gamma}(u)) + h_{\Gamma}(u)I.
\end{equation}
Within this general setting, the matrix $W_{\Gamma}^r(u)$
is no longer guaranteed to be positive semidefinite,
and the associated principal radii of curvature
$\rho_i(\Gamma, u), i = 1, \dots, n-1$, may take negative values.

Let $\mathcal{K}^n(\omega)$ stand for the collection of all $n$-dimensional bodies
of constant width $\omega > 0$. It is known from \cite{H-2006}
that every $n$-dimensional body of constant width is a strictly
compact convex set of $\mathbb{R}^n$.
Consequently, for each convex body $K \in \mathcal{K}^n(\omega)$,
the matrix $W_K^r$ is defined almost everywhere on $S^{n-1}$ and is positive semidefinite.

\vskip.3cm
\begin{lemma}\label{l3.1}
\begin{enumerate}
\item Let $K \in \mathcal{K}^n(\omega)$ be an $n$-dimensional body of constant width
$\omega$, let $h_K(u)$ denote its the support function. Then the following statements hold:
\begin{enumerate}
\item[(i)] $h_K \in C^{1,1}(S^{n-1})$ and $h_K(u) + h_K(-u) = \omega$;

\item[(ii)] If $W_K^r$ exists at $u \in S^{n-1}$, then $0 \leq \rho_i(K, u) \leq \omega, i = 1, \dots, n-1$.
\end{enumerate}

\item Let $h \in C^{1,1}(S^{n-1})$ and set $\Gamma = G^r(h, S^{n-1})$. Suppose that
$h(u) + h(-u) = \omega$ for all $u \in S^{n-1}$, and that
$0 \leq \rho_i(\Gamma, u) \leq \omega$ for all $i = 1, \dots, n-1$ whenever
$W_{\Gamma}^r(u)$ exists.
Then there exists a body of constant width $K \in \mathcal{K}^n(\omega)$
such that $\Gamma = \partial K$.
\end{enumerate}
\end{lemma}

\begin{proof}
(1) From the definition of bodies of constant width together with their strict convexity,
it follows directly that
$h_K \in C^{1,1}(S^{n-1})$ and $h_K(u) + h_K(-u) = \omega$.

Since $W_K^r(u)$ is defined and $\partial K$
is a strictly convex hypersurface, we have
$\rho_i(K, u) \geq 0$ for all $i = 1, \dots, n-1$ and every $u \in S^{n-1}$.
Combining this fact with the constant width condition
$h_K(u) + h_K(-u) = \omega$, we immediately obtain
$$
W_K^r(-u)
= \operatorname{Hess}\bigl(\omega - h_K(u)\bigr)
+ \bigl(\omega - h_K(u)\bigr) I = \omega I - W_K^r(u).
$$
Consequently,
$$
W_K^r(u) + W_K^r(-u) = \omega I.
$$
Because $G_K^r(u)$ and $G_K^r(-u)$ share identical principal directions,
one gets
\begin{equation}\label{e3.2}
\rho_i(K, u) + \rho_i(K, -u) = \omega, \qquad i = 1, \dots, n-1, \quad u \in S^{n-1}.
\end{equation}
Hence assertion (ii) is valid.

(2) Assume that $h \in C^{1,1}(S^{n-1})$ and $h(u) + h(-u) = \omega$
for every $u \in S^{n-1}$. From formula \eqref{e3.1}, it is straightforward to verify that
$\Gamma$ is closed and bounded. Moreover, $W_{\Gamma}^r(u)$ exists almost everywhere on
$S^{n-1}$. Together with the hypothesis $0 \leq \rho_i(\Gamma, u) \leq \omega, i = 1, \dots, n-1$,
this implies that the principal curvatures of $\Gamma$ exist almost everywhere and admit
a positive lower bound. Therefore, $\Gamma$ is a strictly convex
closed bounded hypersurface. Let $K$ denote the compact set enclosed by
$\Gamma$. Then K is a strictly convex body. Hence
$$
h_K(u) = h_{\Gamma}(u) = h(u).
$$
Thus $h_K(u) + h_K(-u) = \omega$, and consequently $K \in \mathcal{K}^n(\omega)$.
\end{proof}

In addition, several geometric properties
of convex bodies of constant width have been studied;
see \cite{LO-2007}.
We collect below a number of facts that will be need for this paper.

\vskip.3cm
\begin{lemma}\label{l3.2}
 Let $K \in \mathcal{K}^n(\omega)$. For any $x \in \partial K$, consider $F_x$,
the set of points in $K$ which are as far as possible from $x \in \partial K$:
$$
F_x := \{x' \in K : |x - x'| = \max_{y \in K} |x - y|\}.
$$
Then for any $x' \in F_x$, we have
$|x - x'| = \omega$
and
$x = G_K^r\!\left(\frac{1}{\omega}(x - x')\right)$.
\end{lemma}

\vskip.3cm
\begin{lemma}\label{l3.3}
 Let $K \subset \mathbb{R}^n$ be a convex body of constant width $\omega$. Then
$$
K = \bigcap_{x \in \partial K} \overline{B^n}(x, \omega),
$$
where $\overline{B^n}(x, \omega)$ is the $n$-dimensional closed ball
of center $x$ and radius $\omega$ in $\mathbb{R}^n$.
\end{lemma}

Let $K \in \mathcal{K}^n(\omega)$ be a convex body of constant width $\omega$.
The two points $G_K^r(u)$ and $G_K^r(-u)$ are said to form a pair of {\it antipodal points} of
$K$ corresponding to the directions $u$ (or $-u$).
Define the {\it midpoint map} $\widetilde{G_K^r}: S^{n-1}\rightarrow \mathbb{R}^n$
by
\begin{equation}\label{e3.3}
\widetilde{G_K^r}(u) = \frac{1}{2} \bigl(G_K^r(u) + G_K^r(-u)\bigr), \quad u \in S^{n-1}.
\end{equation}
Set $\widetilde{\Gamma} = \widetilde{G_K^r}(S^{n-1})$. In other words,
$\widetilde{\Gamma}$ consists of the midpoints of all antipodal pairs of
$K$, taken over every direction. Since $\partial K = G_K^r(S^{n-1})$
is closed and bounded, $\widetilde{\Gamma}$ is likewise
a closed bounded hypersurface in $\mathbb{R}^n$.
This hypersurface is called the {\it Wigner caustic} of
the constant width body $K$.

Define the function $\widetilde{h}(u) = \frac{1}{2}(h_K(u) - h_K(-u))$.
Combining formulas \eqref{e3.1} and \eqref{e3.3} , we obtain
$$
G^r_{\widetilde{\Gamma}}(u) = \widetilde{G_K^r}(u)
= \frac{1}{2} \bigl(G_K^r(u) + G_K^r(-u)\bigr)
= \sum_{i=1}^{n-1} \frac{\partial \widetilde{h}(u)}{\partial x_i} e_i + \widetilde{h}(u)u
= G^r(\widetilde{h}, u).
$$
Thus $\frac{1}{2}(h_K(u) - h_K(-u)) = h_K(u) - \frac{\omega}{2}$
serves as the generalized support function of the Wigner caustic
$\widetilde{\Gamma}$ associated with $K$, and is denoted by $\widetilde{h}_K(u)$.
Define
$$
\widetilde{W_K^r}(u) := W_{{\widetilde\Gamma}}^r(u), \quad
\widetilde{\rho}_i(K, u) := \rho_i(\widetilde{\Gamma}, u), \quad
u \in S^{n-1}, \quad  i = 1, \dots, n-1,
$$
which are, respectively, the inverse Weingarten matrix and
the principal radii of curvature corresponding to the caustic.
The following identities follow directly:
\begin{equation}\label{e3.4}
\begin{aligned}
&G_{\widetilde{\Gamma}}^r(u)=G^r(\widetilde{h}, u) = G^r(\widetilde{h}, -u)=G_{\widetilde{\Gamma}}^r(-u),\\
&\widetilde{W_K^r}(u) = W_K^r(u) - \frac{\omega}{2} I,\\
&\widetilde{\rho}_i(K, u) = \rho_i(K, u) - \frac{\omega}{2}, \qquad i = 1, \dots, n-1.
\end{aligned}
\end{equation}
With these preliminaries established, we now state the general construction
theorem for $n$-dimensional bodies of constant width.

\vskip.3cm
\begin{theorem}\label{t3.4}
Let $\widetilde{h} \in C^{1,1}(S^{n-1})$, and write $\widetilde{\Gamma} = G^r(\widetilde{h}, S^{n-1})$.
Assume the following two conditions are satisfied:

(1) $\widetilde{h}(u) + \widetilde{h}(-u) = 0$ for all $u\in S^{n-1}$;

(2) all second-order partial derivatives of $\widetilde{h}$ are bounded on $S^{n-1}$.

Then there exist a positive constant $\omega$ and a constant width body $K \in \mathcal{K}^n(\omega)$
such that the support function of $K$ satisfies $h_K(u) = \widetilde{h}(u) + \frac{\omega}{2}$.
\end{theorem}

\begin{proof}
By condition (2), the entries of the symmetric matrix $W_{\widetilde{\Gamma}}^r(u)$
are bounded on $S^{n-1}$. Consequently, its eigenvalues $\rho_i(\widetilde{\Gamma}, u), i = 1, \dots, n-1,$
 are bounded over $S^{n-1}$. Choose $\omega > 0$ such that
$$
\sup_{\substack{u \in S^{n-1} \\ i=1,\dots,n-1}} \left|\rho_i(\widetilde{\Gamma}, u)\right| \leq \frac{\omega}{2}.
$$
Set $h(u) = \widetilde{h}(u) + \frac{\omega}{2}$ and $\Gamma = G^r(h, S^{n-1})$, for $u \in S^{n-1}$.
Combining condition (1) with Lemma \ref{l3.1}, the closed set enclosed by $\Gamma$
is precisely the desired convex body $K$, whose support function is $h$ and which has constant width $\omega$.
This completes the proof of Theorem \ref{t3.4}.
\end{proof}

\section{Necessary Conditions for the $n$-Dimensional Blaschke-Lebesgue Problem}\label{s4}
\setcounter{equation}{0}

Let $K, L$ be convex bodies in $\mathbb{R}^n$ and $\lambda \in [0, 1]$,
the classical Brunn-Minkowski inequality states that (see \cite{S-2014}):
$$
V^{\frac{1}{n}}\bigl(\lambda K + (1 - \lambda)L\bigr)
\geq \lambda V^{\frac{1}{n}}(K) + (1 - \lambda)V^{\frac{1}{n}}(L).
$$
Equality holds for $\lambda \in (0, 1)$ if and only if $K$ and $L$ are homothetic. Here $K + L$
stands for the Minkowski sum of $K$ and $L$, whose support function obeys the additivity rule
$$
h_{K+L}(u) = h_K(u) + h_L(u).
$$
Consequently, for any two convex bodies $K, L \in \mathcal{K}^n(\omega)$,
the convex combination $\lambda K + (1-\lambda)L$
is again a body of constant width $\omega$ for each
$\lambda \in [0, 1]$. Accordingly, we define the set
$$
\mathcal{H}^n(\omega) := \bigl\{ h_K(u) : u \in S^{n-1} \mid K \in \mathcal{K}^n(\omega) \bigr\},
$$
which forms a convex subset of the function space $C^{1,1}(S^{n-1})$.

On the other hand, formula \eqref{e2.5} shows that the volume functional $V$
is homogeneous of degree $n$ on $C^{1,1}(S^{n-1})$.
Hence the $n$-dimensional Blaschke-Lebesgue problem can be restated as
a nonlinear convex optimization problem:
minimize the $n$-homogeneous volume functional $V$ subject to the constraint
that the support function belongs to the convex set $\mathcal{H}^n(\omega)$.

From the Brunn-Minkowski inequality one deduces that, for arbitrary $K, L \in \mathcal{K}^n(\omega)$ and
$\lambda \in [0, 1]$, the functional $-V^{\frac{1}{n}}$ satisfies the subadditivity relation
$$
-V^{\frac{1}{n}}\bigl(\lambda K + (1 - \lambda)L\bigr)
\leq \lambda\bigl(-V^{\frac{1}{n}}(K)\bigr) + (1 - \lambda)\bigl(-V^{\frac{1}{n}}(L)\bigr).
$$
This shows that $-V^{\frac{1}{n}}$ is a convex functional over the convex set
$\mathcal{H}^n(\omega)$. Minimizing the volume $V$ is therefore equivalent to
maximizing the functional $-V^{\frac{1}{n}}$. By standard
properties of convex functionals, any maximizer attained on $\mathcal{H}^n(\omega)$
must lie on its boundary. The geometric property
corresponding to the boundary of $\mathcal{H}^n(\omega)$ is given in the following lemma.

\vskip.3cm
\begin{lemma}\label{l4.1}
Let $K \in \mathcal{K}^n(\omega)$. If $K$
is a minimizer for the $n$-dimensional Blaschke-Lebesgue problem, then
$$
\max \bigl\{ |\widetilde{\rho}_i(K, u)| : i = 1, \dots, n-1 \bigr\} = \frac{\omega}{2}
$$
holds almost everywhere on $S^{n-1}$, where
$
\widetilde{\rho}_i(K, \cdot)$ denote the principal radii of curvature of the Wigner caustic of $K
$.
\end{lemma}
	
\begin{proof}
 We proceed by contradiction.
Suppose that $K \in \mathcal{K}^n(\omega)$
 is a volume minimizer for the $n$-dimensional Blaschke-Lebesgue problem.
 Assume there exists a measurable set $E \subset S^{n-1}$ of positive Lebesgue measure such that
\begin{equation}\label{e4.1}
\sup_{\substack{u \in E \\ i=1,\dots,n-1}} |\widetilde{\rho}_i(K, u)| < \frac{\omega}{2}.
\end{equation}
Since $E$ has positive measure,
we can pick a point $u_0$ lying in the interior of $E$. Define
\begin{equation*}
\begin{aligned}
&S_1 = \bigl\{ u : |u - u_0| < r_1, \; u \in E \bigr\},
\\
&S_2 = \bigl\{ u : |u - u_0| < r_2, \; u \in E \bigr\},
\end{aligned}
\end{equation*}
and
$$
S_+ = \bigl\{ u : \langle u, u_0 \rangle \geq 0, \; u \in S^{n-1} \bigr\},
$$
where the real numbers $r_1, r_2$ satisfy $0 < r_1 < r_2 < \sqrt{2}$. Then
$$
S_1 \subsetneq S_2 \subsetneq S_+.
$$
We construct a smooth cutoff function $c(u)$ on $S^{n-1}$ by
$$
c(u) =
\begin{cases}
1, & u \in S_1, \\[4pt]
G(|u - u_0|), & u \in S_2 \setminus S_1, \\[4pt]
0, & u \in S_+ \setminus S_2, \\[4pt]
-c(-u), & u \in S^{n-1} \setminus S_+,
\end{cases}
$$
with
$$
G(x) = \frac{\int_{x}^{r_2} g(t)\,dt}{\int_{r_1}^{r_2} g(x)\,dx}, \qquad g(x) = e^{\frac{1}{(x-r_1)(x-r_2)}}.
$$
Let $S_2^- = \{-u : u \in S_2\}$.
From the definition of the cutoff function $c(u)$,
one readily checks that
$$
c(u) = 0, \qquad u \in S^{n-1} \setminus \bigl(S_2 \cup S_2^-\bigr).
$$
Define
\begin{equation}\label{e4.2}
h_1(u) = h_K(u) + \varepsilon c(u), \qquad h_2(u) = h_K(u) - \varepsilon c(u),
\end{equation}
where $\varepsilon > 0$ is a positive constant.
Let $\Gamma_1, \Gamma_2$ denote the closed bounded hypersurfaces
with generalized support functions $h_1, h_2$, respectively.
Since the cutoff function and all its derivatives of arbitrary order are bounded,
for $u \in S_2 \subset E$, in view of equations \eqref{e4.1} and \eqref{e4.2},
there exists $\varepsilon > 0$ sufficiently small such that
$$
\left|\rho_i(\Gamma_j, u) - \frac{\omega}{2}\right| \leq \frac{\omega}{2},
\quad i = 1, \dots, n-1, \quad j = 1, 2.
$$
Moreover,
$$
h_j(u) + h_j(-u) = \omega, \quad u \in S^{n-1}, \quad j = 1, 2.
$$
By Lemma \ref{l3.1}, there exist constant width bodies $K_1, K_2 \in \mathcal{K}^n(\omega)$
satisfying $\partial K_1 = \Gamma_1$, $\partial K_2 = \Gamma_2$,
whose support functions are given by
$$
h_{K_1} = h_1 = h_K + \varepsilon c, \qquad h_{K_2} = h_2 = h_K - \varepsilon c.
$$
Consequently, $2K = K_1 + K_2.$

We claim that $K_1$ and $K_2$ cannot be homothetic.
Suppose, for contradiction, that $K_1$ and $K_2$ are homothetic.
Then there exist a constant vector $a \in \mathbb{R}^n$
and a real constant $\mu \in \mathbb{R}$ such that $K_1 = a + \mu K_2$, or
equivalently,
$$
h_1(u) = \langle a, u \rangle + \mu h_2(u), \qquad u \in S^{n-1}.
$$
For $u \in S^{n-1} \setminus (S_2 \cup S_2^-)$, we have $c(u) = 0$,
hence $h_1(u) = h_2(u) = h_K(u)$. This yields
\begin{equation}\label{e4.3}
(1 - \mu)h_K(u) = \langle a, u \rangle.
\end{equation}

If $\mu \neq 1$, then
$
h_K(u) = \left\langle \frac{1}{1-\mu} a, u \right\rangle.
$
Consequently,
$$
0 < \omega = h_K(u) + h_K(-u)
= \left\langle \frac{1}{1-\mu} a, u \right\rangle + \left\langle \frac{1}{1-\mu} a, -u \right\rangle
= 0,
$$
which leads to a contradiction.

If $\mu = 1$, then from \eqref{e4.3} we obtain $a = 0$. Hence
$h_1(u) = h_2(u)$ for all $u \in S^{n-1}$. Nevertheless, by
construction of $c(u)$, we have $c(u) = 1$ for $u \in S_1$,
which implies $h_1(u) \neq h_2(u)$, again a contradiction.

Therefore, $K_1$ and $K_2$ are not homothetic.
Combining this fact with the Brunn-Minkowski inequality, we obtain
$$
2V^{\frac{1}{n}}(K) > V^{\frac{1}{n}}(K_1) + V^{\frac{1}{n}}(K_2),
$$
so that at least one of $K_1, K_2$ has strictly smaller volume than $K$.
This contradicts the hypothesis that $K$ is a volume minimizer for the
Blaschke-Lebesgue problem.
Therefore, no measurable set $E\subset S^{n-1}$ of positive Lebesgue measure
can satisfy
$$
\sup_{\substack{u \in E \\ i=1,\dots,n-1}} |\widetilde{\rho}_i(K, u)| < \frac{\omega}{2}.
$$
Together with Lemma \ref{l3.1} and equation \eqref{e3.4}, we conclude that
$$
\max \bigl\{ |\widetilde{\rho}_i(K, u)| : i = 1, \dots, n-1 \bigr\} = \frac{\omega}{2}
$$
holds almost everywhere on $S^{n-1}$.
\end{proof}

If $K \in \mathcal{K}^n(\omega)$
is a minimizer for the $n$-dimensional Blaschke-Lebesgue problem,
then by Lemma \ref{l4.1} together with equations \eqref{e3.2} and \eqref{e3.4},
for almost every $u\in S^{n-1}$
there exists at least one principal direction $\xi_{i_0} \in T_u S^{n-1}$,
with $i_0 \in \{1, \dots, n-1\}$,
such that the corresponding principal radii of curvature satisfy
$$
\{\rho_{i_0}(K, u), \rho_{i_0}(K, -u)\} = \{0, \omega\}.
$$
When $\rho_{i_0}(K, u) = 0$, the inverse Weingarten matrix $W_K^r(u)$
becomes degenerate, and accordingly $G_K^r(u)$ is a non-regular point of $\partial K$.
From this observation we derive the geometric property of
minimizers for the Blaschke-Lebesgue problem,
formulated in the following theorem.

\vskip.3cm
\begin{theorem}\label{t4.2}
(Necessary conditions for the n-dimensional Blaschke-Lebesgue problem)
Let $K \in \mathcal{K}^n(\omega)$
be a volume minimizer for the $n$-dimensional Blaschke-Lebesgue problem.
Then, for almost every pair of antipodal points of $K$, one point is
non-regular; at the other point, the minimum principal curvatures equals $\frac{1}{\omega}$.
\end{theorem}

\vskip.3cm
\begin{definition}\label{d4.3}
Let $K\in \mathcal{K}^n(\omega)$, we say that $K$ possesses {\it property} P,
if it satisfies the necessary condition given in Theorem \ref{t4.2},
that is, for almost every pair of antipodal points
of $K$, one point is non-regular, and at the other point the minimum
principal curvatures equals $\frac{1}{\omega}$.
Denote by $\mathcal{K}_P^n(\omega)$ the collection of all bodies
of constant width $\omega$ that satisfy Property P.
\end{definition}

\noindent {\it Remark 2.}~
Theorem \ref{t4.2} is a higher-dimensional generalization of Theorem A.
On the one hand, by the definition of non-regular points,
we may classify non-regular points according to
the rank of the inverse Weingarten matrix.
For instance, let $K \in \mathcal{K}^n(\omega)$
be a minimizer for the $n$-dimensional Blaschke-Lebesgue problem.
If $\partial K$ is non-regular at $G_K^r(u)$,
then $\operatorname{rank}(W_K^r(u)) < n-1$. When $n \geq 2$ and
$\operatorname{rank}(W_K^r(u)) = 0$, the tangent space $T_{G_K^r(u)} \partial K$
is the zero-dimensional subspace of $T_u S^{n-1}$, the local degree of freedom of
$G_K^r(u)$ equals $0$, so $G_K^r(u)$ is a vertex of $\partial K$.
When $n \geq 3$ and $\operatorname{rank}(W_K^r(u)) = 1$, $T_{G_K^r(u)} \partial K$
is a one-dimensional subspace of $T_u S^{n-1}$, the local degree of freedom of
$G_K^r(u)$ is $1$, so $G_K^r(u)$ lies on a one-dimensional edge of $\partial K$.

On the other hand, $\partial K$ consists of regular points almost everywhere,
meaning that the set of all non-regular points forms a null set in $\partial K$.
Therefore, Theorem \ref{t4.2} can be equivalently states that
$\partial K$ is almost everywhere composed of regular points whose
minimum principal curvature equals $\frac{1}{\omega}$,
and the corresponding antipodal points constitute precisely
the non-regular part of $\partial K$. In particular, for $n = 3$,
the non-regular points of $\partial K$ are either vertices or
lie on edges, and the minimum
principal curvature at each regular point is $\frac{1}{\omega}$.
This coincides with the result of Theorem A.

Using Theorem \ref{t4.2}, we can give a new proof of the 2-dimensional Blaschke-Lebesgue theorem:

\vskip.3cm
\begin{corollary}\label{c4.4}(The Blaschke-Lebesgue theorem)
Among all planar bodies of constant width $\omega > 0$,
the Reuleaux triangle has minimal area.
\end{corollary}

\begin{proof}
Let $K \in \mathcal{K}^2(\omega)$ be a planar body of constant
width with minimal area. Its boundary $\partial K$ is a closed convex
curve in $\mathbb{R}^2$.
Every point on this curve has a unique tangent direction,
which yields a uniquely defined curvature.
By Theorem \ref{t4.2}, the radius of curvature $\rho$ at each boundary point
can take only the values $0$ or $\omega$:
\begin{itemize}
\item[\upshape(i)] when $\rho = 0$, the curvature $\kappa$ satisfies $\kappa=\infty$,
and the corresponding boundary point is a non-regular cusp;
\item[\upshape(ii)] when $\rho = \omega$, the curvature $\kappa$ satisfies $\kappa=\frac{1}{\omega}$,
and the corresponding boundary point lies on a smooth circular arc of radius $\omega$.
\end{itemize}
To summarize the above discussion,
$\partial K$ consists of only two components:
circular arc segments with constant curvature $\frac{1}{\omega}$,
and non-regular cusps with infinite curvature.
The antipodal point of every point on these circular arcs
is exactly a non-regular cusp.
Moreover, there exists no curve segment on $\partial K$
whose curvature is continuous but non-constant curvature.
The planar constant-width bodies
satisfying these geometric features are precisely the
Reuleaux polygons. Applying the result due to Firey
in reference \cite{F-1960}, which asserts that the Reuleaux triangle attains
the minimal area among all Reuleaux polygons of equal width,
we conclude that, for fixed width, the Reuleaux triangle
has the minimal area among all planar bodies of constant width.
\end{proof}

\section{Dimension-Raising Construction of Lachand-Robert and Oudet}\label{s5}
\setcounter{equation}{0}

We first recall the {\it dimension-raising construction} proposed
by Lachand-Robert and Oudet in \cite{LO-2007},
which produces an $n$-dimensional body of constant width from any
$(n-1)$-dimensional body of constant width. The construction is stated as follows.

Let $H \subset \mathbb{R}^n$ be an affine hyperplane,
$E_+^n$ and $E_-^n$ the two open half-spaces separated by $H$,
and $K_0 \subset H$ be an $(n-1)$-dimensional body of constant width $\omega > 0$.
Lachand-Robert and Oudet constructed the $n$-dimensional body
$K^n$ of constant width $\omega$ by
\begin{equation}\label{e5.1}
\begin{aligned}
K^n \cap \overline{E_+^n} = K_+^n := \overline{E_+^n} \cap \bigcap_{x \in K_0} \overline{B^n}(x, \omega),
\end{aligned}
\end{equation}
\begin{equation}\label{e5.2}
\begin{aligned}
K^n \cap E_-^n = K_-^n := E_-^n \cap \bigcap_{x \in K_+^n} \overline{B^n}(x, \omega).
\end{aligned}
\end{equation}

Let $K$ be an $m$-dimensional body of constant width $\omega$,
and write $K^m=K$ for the initial convex body. Iterating the
above Lachand-Robert-Oudet dimension-raising construction
yields a sequence of convex bodies $\{K^n : n \geq m\}$ starting from $K$,
which we denote by $\mathcal{LO}(K)$. Clearly, for every $K^n \in \mathcal{LO}(K)$,
we have $K^n \in \mathcal{K}^n(\omega)$.

\vskip.3cm
\begin{theorem}\label{t5.1}
Let $K \in \mathcal{K}^m(\omega)$. If $K^n \in \mathcal{LO}(K)$, $n \geq m$,
then $K^n \in \mathcal{K}_P^n(\omega)$, that is,
$K^n$ is a body of constant width satisfying Property P.
\end{theorem}

\begin{proof}
By Lemma \ref{l3.2}, for any $x, y \in K \in \mathcal{K}^n(\omega)$,
we have $|y - x| \leq \omega$, with equality if and only if
$x, y$ are a pair of antipodal points in the direction $\frac{1}{\omega}(y - x)$.

We claim the following: for every $K \in \mathcal{K}^n(\omega)$ and for any
$
x, y_1, y_2 \in \partial K$,
if
$y_1 \neq y_2$
and
$|y_1 - x| = |y_2 - x| = \omega$,
then the minor arc $s(y_1, y_2)$ centered at $x$ with radius
$\omega$ and endpoints $y_1, y_2$, given by
\begin{equation}\label{e5.3}
s(y_1, y_2) = \left\{ x + \omega u \mid u
= \lambda u_1 + \mu u_2 \in S^{n-1},\ \lambda, \mu \geq 0
 \right\},
\end{equation}
satisfies $s(y_1, y_2) \subset \partial K$, where
$u_1 = \frac{1}{\omega}(y_1 - x),\ u_2 = \frac{1}{\omega}(y_2 - x)$.

Indeed, suppose $s(y_1, y_2) \not\subset \partial K$.
Then there exists at least one point $y_0 = x + \omega u_0 \in s(y_1, y_2)$
such that $y_0 \notin \partial K$, where $u_0 = \lambda_0 u_1 + \mu_0 u_2$
 with $\lambda_0, \mu_0 > 0$. Since the diameter of the constant width body
 $K$ equals $\omega$, it follows that $y_0 \notin K$.
 By Lemma \ref{l3.3}, there exists $z\in K$ such that $|z-y_0|>\omega$.
 Set $a=z-x$, it follows from $x,y_1,y_2,z\in K$ that
$$
|z - x| \leq \omega, \qquad |z - y_1| \leq \omega, \qquad |z - y_2| \leq \omega,
$$
that is,
$$a^2 \leq \omega^2, \qquad (a - \omega u_1)^2 \leq \omega^2, \qquad (a - \omega u_2)^2 \leq \omega^2.$$
From these we obtain
$$
2\omega\langle a, u_1\rangle \geq a^2, \qquad 2\omega\langle a, u_2\rangle \geq a^2.
$$
Moreover, since
$
\omega^2 < (z - y_0)^2 = (a - \omega u_0)^2,
$
we obtain
$$
a^2 > 2\omega\langle a, \lambda_0 u_1\rangle + 2\omega\langle a, \mu_0 u_2\rangle \geq (\lambda_0 + \mu_0)a^2.
$$
Hence
$\lambda_0 + \mu_0 = |\lambda_0 u_1| + |\mu_0 u_2| < 1$.
However, $|u_0| = |\lambda_0 u_1 + \mu_0 u_2| = 1$,
which yields a contradiction.
Therefore, $s(y_1, y_2) \subset \partial K$.

If $y_1 \in \partial K^n_+$
and
$y_1 \notin K^{n-1}$,
then by \eqref{e5.1}, there exists
$x \in \partial K^{n-1} \subset \partial K^n$ such that $|y_1 - x| = \omega$.
Let $y_2 \in \partial K^{n-1}$ be an antipodal point of $ x$ on
$K^{n-1}$, so that $|y_2 - x| = \omega$.

By \eqref{e5.3}, the circular arc $s(y_1, y_2)$ centered at $x$
with radius $\omega$ and endpoints $y_1, y_2$ is contained in
$\partial K^n$. Since the curvature of $s(y_1, y_2)$ is constantly
$\frac{1}{\omega}$, by Lemma \ref{l3.1} and equation \eqref{e3.2},
whenever $y_1 = G_{K^n}(u_1) \in \partial K^n_+$ is a regular point,
it follows from $y_1 \in s(y_1, y_2) \subset \partial K^n$
that the maximum principal radius of curvature of $\partial K^n$ at $y_1$
satisfies $\rho_{i_0}(u_1) = \omega$, and the minimum principal curvature
 equals $\frac{1}{\omega}$. Moreover, at the antipodal point $G_{K^n}(-u_1)$,
 the principal radius of curvature in the corresponding principal direction satisfies
 $\rho_{i_0}(-u_1) = 0$, so this point is non-regular. Conversely,
 if $y_1 \in \partial K^n_+$ is non-regular, then the minimum
 principal curvature at its antipodal point is $\frac{1}{\omega}$.
 By the dimension-raising construction, the antipodal point of any point of $\partial K_-^n$
 belongs to $\partial K^n_+$. Therefore, $K^n$ satisfies Property P,
 that is, for every $K^n \in \mathcal{LO}(K)$ with $n > m$, we have
 $K^n \in \mathcal{K}_P^n(\omega)$.
\end{proof}

In particular, when $K$ is a line segment of length $\omega > 0$,
Theorem \ref{t5.1} implies that if $K^n \in \mathcal{LO}(K)$ with $n \geq 1$,
then necessarily $K^n \in \mathcal{K}_P^n(\omega)$. Here, $K^1 = K$
 is the unique one-dimensional body of constant width in $\mathcal{K}^1(\omega)$;
 $K^2$ is the Reuleaux triangle, the minimizer for the
 planar Blaschke-Lebesgue problem; and $K^3$ is the Meissner body,
 which is conjectured to be the minimizer for the
 three-dimensional Blaschke-Lebesgue problem.
 Based on Theorems \ref{t4.2} and \ref{t5.1}, we formulate the following conjecture.

{\bf Conjecture}\quad
 Let $\omega>0$ be a real number,
 let $K$ be a line segment of length $\omega$,
 and let $\mathcal{LO}(K)$ denote the sequence of
 constant width convex bodies obtained from $K$
 via the dimension-raising construction.
 Then $\mathcal{LO}(K)$ yields the family of minimizers for the
 Blaschke-Lebesgue problem. that is, whenever $K^n \in \mathcal{LO}(K)$,
 $K^n$ is a minimizer for the corresponding
 $n$-dimensional Blaschke-Lebesgue problem.

{\bf Funding information}
This work is supported by the Excellent Young Talents Fund Program of
Higher Education Institutions of Anhui Province(No. gxyqZD2020022) and
the Horizontal Scientific Research Funds
in Huaibei Normal University (No. 2025340609000006).

{\bf Conflict of Interest}
The corresponding author states that there is no conflict of interest.

{\bf Data availability statement}
Data sharing is not applicable to this article as no new data were created
 or analyzed in this study.

\end{document}